\documentclass[11pt]{amsart}

\usepackage{pgfplots}
\pgfplotsset{compat=1.15}
\usepackage{mathrsfs}
\usetikzlibrary{arrows}

\usepackage[utf8]{inputenc}
\usepackage{mathpazo, eulervm}
\usepackage{subcaption}
\usepackage{amsthm, amsmath, amssymb}
\usepackage[a4paper, margin=1in, top=3cm]{geometry}
\usepackage{tikz}
\usetikzlibrary{cd}
\usepackage{hyperref}
\usepackage{bbding}
\usepackage{bm}
\usepackage{mathtools}

\hypersetup{
  colorlinks=true,
  citecolor=red,
  urlcolor=blue,
  linkcolor=violet
}

\newtheorem{theorem}{Theorem}[section]
\newtheorem{lemma}[theorem]{Lemma}
\newtheorem{proposition}[theorem]{Proposition}
\newtheorem{corollary}[theorem]{Corollary}

\theoremstyle{definition}
\newtheorem{definition}[theorem]{Definition}
\newtheorem{construction}[theorem]{Construction}
\newtheorem{example}[theorem]{Example}
\newtheorem{remark}[theorem]{Remark}
\newtheorem{question}[theorem]{Question}

\DeclarePairedDelimiter{\floor}{\lfloor}{\rfloor}
\DeclarePairedDelimiter{\abs}{\lvert}{\rvert}

\newcommand{\NN}{\mathbb{N}}

\newcommand{\ZZ}{\mathbb{Z}}

\newcommand{\RR}{\mathbb{R}}

\newcommand{\bfx}{\mathbf{x}}

\newcommand{\bfy}{\mathbf{y}}
\newcommand{\bfa}{\mathbf{a}}
\newcommand{\bfb}{\mathbf{b}}

\newcommand{\ub}{\mathbf{u}}
\newcommand{\vb}{\mathbf{v}}
\newcommand{\eb}{\mathbf{e}}
\newcommand{\tb}{\mathbf{t}}
\newcommand{\kk}{\mathbf{k}}
\newcommand{\mm}{\mathbf{m}}

\newcommand{\set}[1]{\left\{ #1 \right\}}
\newcommand{\setcond}[2]{\set{#1 \ \colon \ #2}}

\newcommand{\facets}{\mathcal{F}}

\newcommand{\Nc}{\mathcal{N}}

\newcommand{\boundary}{\partial}
\newcommand{\ant}{\mathrm{ant}}
\newcommand{\trace}{\mathrm{tr}}
\newcommand{\tr}{\mathrm{tr}}
\newcommand{\conv}{\mathrm{conv}}

\newcommand{\Hom}{\mathrm{Hom}}
\newcommand{\verto}{\mathrm{vert}}
\newcommand{\into}{\mathrm{int}}
\newcommand{\strint}{\mathrm{int}}

\newcommand{\Stab}{\mathrm{Stab}}

\begin{document}


\author[T.\,Hall]{Thomas~Hall}
\address[T.\,Hall]{School of Mathematical Sciences\\University of Nottingham\\United Kingdom}
\email{pmyth1@nottingham.ac.uk}

\author[M.\,K{\" o}lbl]{Max~K{\" o}lbl}
\address[M.\,K{\" o}lbl]{Department of Pure and Applied Mathematics, Graduate School of Information Science and Technology, Osaka University, Suita, Osaka 565-0871, Japan}
\email{max.koelbl@ist.osaka-u.ac.jp}

\author[K.\,Matsushita]{Koji~Matsushita}
\address[K.\,Matsushita]{Department of Pure and Applied Mathematics, Graduate School of Information Science and Technology, Osaka University, Suita, Osaka 565-0871, Japan}
\email{k-matsushita@ist.osaka-u.ac.jp}

\author[S.\,Miyashita]{Sora~Miyashita}
\address[S.\,Miyashita]{Department of Pure and Applied Mathematics, Graduate School of Information Science and Technology, Osaka University, Suita, Osaka 565-0871, Japan}
\email{u804642k@ecs.osaka-u.ac.jp}

\title{Nearly Gorenstein Polytopes}
\maketitle

\begin{abstract}
    In this paper, we study nearly Gorensteinness of Ehrhart rings arising from lattice polytopes.
    We give necessary conditions and sufficient conditions on lattice polytopes for their Ehrhart rings to be nearly Gorenstein.
    Using this, we give an efficient method for constructing nearly Gorenstein polytopes.
    Moreover, we determine the structure of nearly Gorenstein $(0,1)$-polytopes and characterise nearly Gorensteinness of edge polytopes and graphic matroids.
\end{abstract}

\section{Introduction}
Let $\kk$ be an infinite field, and let us denote the set of nonnegative integers, the set of integers, the set of rational numbers and the set of real numbers by $\mathbb{N}$, $\mathbb{Z}$, $\mathbb{Q}$ and $\mathbb{R}$, respectively.

Let $P \subset \RR^d$ be a \emph{lattice} polytope, which is a convex polytope whose vertices all have integer coordinates.
If we place $P$ at height $1$ in $\RR^{d+1}$ and take the cone over it, we obtain $C_P \subset \RR^{d+1}$.
The \emph{Ehrhart ring} $A(P)$ of $P$ is defined by $\kk[C_P \cap \ZZ^{d+1}]$, where each lattice point $(x_1, \ldots, x_d, k) \in \ZZ^{d+1}$ is identified with a Laurent monomial $t_1^{x_1} \cdots t_d^{x_d} s^k$.
This classical construction allows for the study of ring theoretic notions via polytopes and combinatorics, and vice versa.

Cohen-Macaulay rings and Gorenstein rings play a central role in commutative algebra.
In the study of rings which are Cohen-Macaulay but not Gorenstein, it has been useful to water down the strong property of being Gorenstein; in fact, many generalised notions of Gorensteinness have been explored.
There are \emph{nearly Gorenstein} rings, \emph{level} rings, and \emph{almost Gorenstein} rings, to name just a few examples.

In this paper, we primarily focus on the nearly Gorenstein property, as introduced in \cite{herzog2019trace}.
Let $R$ be a Cohen-Macaulay ring which is a finitely generated $\NN$-graded $\kk$-algebra.
The definition of nearly Gorenstein arises from studying the non-Gorenstein locus of $R$, which is determined by the \emph{trace} $\tr(\omega_R)$ of the canonical module $\omega_R$ of $R$ (see Definition~\ref{def:trace}).
Explicitly, $R$ is Gorenstein if and only if this trace coincides with the ring itself, i.e. $\tr(\omega_R) = R$.
We call $R$ \emph{nearly Gorenstein} if this trace contains the (unique) maximal graded ideal $\mm$ of $R$, i.e. $\mm \subseteq \tr(\omega_R)$.

Recently, the nearly Gorenstein property has been studied for certain special cases, such as Hibi rings \cite[Theorem 5.4]{herzog2019trace},
edge rings associated to edge polytopes \cite{hibi2021nearly}, numerical semigroup rings \cite{herzog2021canonical}, and projective monomial curves \cite{miyashita2023nearly}.
Moreover, $h$-vectors of nearly Gorenstein homogeneous affine semigroup rings are also studied \cite[Theorem 4.4]{miyashita2022levelness}.

It is a classical result that the lattice polytopes whose Ehrhart rings are Gorenstein are those for which there exists an integer $k$ such that $kP$ is reflexive \cite{batyrev2008}, after an appropriate translation.
In this paper, we study the nearly Gorensteinness of the Ehrhart rings arising from general lattice polytopes.

In Section~\ref{sec:prelim}, we detail the important definitions and results concerning nearly Gorenstein $\kk$-algebras.
We then provide details on Ehrhart rings of lattice polytopes.

In Section~\ref{sec:nearly}, we discuss some relations between nearly Gorensteinness of Ehrhart rings and their polytopes.
We denote the natural pairing between an element $n \in (\RR^d)^*$ and an element $x \in \RR^d$ by $n(x)$.
Let $P \subset \RR^d$ be a lattice polytope and $\facets(P)$ be the set of facets of $P$.
We fix its facet presentation:
\[
    P = \setcond{x\in \RR^d}{n_F(x) \ge -h_F \text{ for all $F\in \facets(P)$}},
\]
where each height $h_F$ is an integer and each inner normal vector $n_F \in (\ZZ^d)^*$ is a \emph{primitive} lattice point, i.e. a lattice point such that the greatest common divisor of its coordinates is $1$.

For a lattice polytope $P \subset \RR^d$, we define its \emph{floor polytope} as $\floor{P} \coloneqq \conv(\strint(P) \cap \ZZ^d)$.
We also introduce the \emph{remainder polytope} $\set{P}$ of $P$, whose definition involves the pushing in/out of its facets in a particular way (see Definition~\ref{def:floor_rem} for the explicit details).
These polytopes are central to our study of nearly Gorenstein polytopes.
Also of importance is the \emph{codegree} $a_P$ of a lattice polytope $P$, which is defined as $a_P \coloneqq \min\setcond{k \in \NN}{\strint(kP) \cap \ZZ^d \neq \varnothing}$, i.e. the minimum positive integer you have to dilate $P$ by until its interior contains lattice points \cite{Batyrev2006}.

\bigskip

We now give the main results of Section~\ref{sec:nearly}.
Our first theorem gives a necessary condition and a sufficient condition for a lattice polytope to be nearly Gorenstein.

\begin{theorem}[{Proposition~\ref{prop:dec}} and Theorem~\ref{thm:big_enough_nG}]\label{thm:nec_suf_intro}
    Let $P \subset \RR^d$ be a lattice polytope with codegree $a$.
    \begin{enumerate}
        \item If $P$ is nearly Gorenstein, then it has the Minkowski decomposition $P = \floor{aP} + \set{P}$.
        \item Conversely, if $P = \floor{aP} + \set{P}$, then there exists some $K$ such that, for all integers $k \ge K$, the polytope $kP$ is nearly Gorenstein.
    \end{enumerate}
\end{theorem}

The next main theorem gives facet presentations for the floor and remainder polytopes appearing in the Minkowski decomposition of a nearly Gorenstein polytope.

\begin{theorem}[{Theorem~\ref{thm:conj1}}]\label{thm:conj1_intro}
    Let $P \subset \RR^d$ be a lattice polytope with codegree $a$.
    Suppose that $P = \floor{aP} + \set{P}$.
    Then
    \begin{align*}
        \floor{aP} &= \setcond{x \in \RR^d}{n_F(x) \ge 1 - ah_F \text{ for all } F \in \facets(P)} \text{ and } \\
        \set{P} &= \setcond{x \in \RR^d}{n_F(x) \ge (a-1)h_F - 1 \text{ for all } F \in \facets(P)}.
    \end{align*}
    Furthermore, if $\floor{P} \neq \varnothing$, then $\set{P}$ is reflexive.
\end{theorem}

These results allow us to prove the final main theorem of Section~\ref{sec:nearly}.
It reveals that the primitive inner normal vectors of a nearly Gorenstein polytope come from boundary points of reflexive polytopes.

\begin{theorem}\label{thm:normal_fan_intro}
    Let $P \subset \RR^d$ be a nearly Gorenstein polytope.
    Then there exists a reflexive polytope $Q \subset \RR^d$ such that
    \[
        P = \setcond{x \in \RR^d}{n(u) \ge -h_n \text{ for all } n \in \boundary Q^* \cap (\ZZ^d)^*},
    \]
    where $h_n$ are integers.
    Moreover, the inequalities defined by $n \in \verto(Q^*)$ are irredundant.
    Furthermore, the number of facets of a nearly Gorenstein polytope is bounded by a constant depending on the dimension $d$.
\end{theorem}

We then use Theorem~\ref{thm:normal_fan_intro} to derive an efficient method for constructing nearly Gorenstein polytopes.
Using this method, we find an example of a nearly Gorenstein polytope which does not have a Minkowski decomposition into Gorenstein polytopes (Example~\ref{exa:counter}).
We conclude the section by studying Minkowski indecomposable nearly Gorenstein polytopes; in particular, we show that they are in fact Gorenstein.

In Section~\ref{sec:(0,1)}, we study nearly Gorenstein $(0,1)$-polytopes.
This family of polytopes includes many subfamilies of polytopes which arise in combinatorics, such as order polytopes of posets and base polytopes from graphic matroids.
Previous work has studied nearly Gorensteinness of Hibi rings \cite{herzog2019trace} and of Ehrhart rings of stable set polytopes arising from perfect graphs \cite{hibi2021nearly,miyazaki2022h-perfect}.
The main result of this section generalises these previous results by characterising a large class of nearly Gorenstein $(0,1)$-polytopes:

\begin{theorem}[Theorem~\ref{01}]\label{01_intro}
    Let $P$ be a $(0,1)$-polytope which has the integer decomposition property.
    Then, $P$ is nearly Gorenstein if and only if $P = P_1 \times \cdots \times P_s$, for some Gorenstein $(0,1)$-polytopes $P_1, \ldots, P_s$ which satisfy $|a_{P_i} - a_{P_j}| \le 1$, for $1 \le i < j \le s$.
\end{theorem}

In Subsection~\ref{subsec:char}, we go into more detail how Theorem~\ref{01_intro} extends previous results which appear in the literature.
Subsequently, we obtain a number of our own interesting corollaries from Theorem~\ref{01_intro}.
For example, we show that every nearly Gorenstein $(0,1)$-polytope which has the integer decomposition property is level (Corollary~\ref{01NGimpliesL}).
Furthermore, we characterise nearly Gorenstein edge polytopes which have the integer decomposition property (Corollary~\ref{cor:edge_polytopes}) and nearly Gorenstein base polytopes arising from graphic matroids (Corollary~\ref{cor:graphic_matroids}).

\subsection*{Acknowledgements}
The first author was supported by JSPS Predoctoral fellowship (short-term) PE22729 while undertaking this work and would like to thank his supervisors Alexander Kasprzyk and Johannes Hofscheier for their useful comments during the write up.
The third author is partially supported by Grant-in-Aid for JSPS Fellows Grant JP22J20033.
We are grateful to Professor Akihiro Higashitani for his very helpful comments and instructive discussions.

\section{Preliminaries and auxiliary lemmas}\label{sec:prelim}
\subsection{Nearly Gorenstein $\kk$-algebras}
Let $R$ be a finitely generated $\NN$-graded $\kk$-algebra with unique graded maximal ideal $\mm$.
We will always assume that $R$ is Cohen-Macaulay and admits a canonical module $\omega_R$.
We call $a(R)$ the $a$-invariant of $R$, i.e.
$$a(R) = -\min\setcond{i \in \NN}{(\omega_R)_i \neq 0},$$ where $(\omega_R)_i$ is the $i$-th graded piece of $\omega_R$.

\begin{definition}\label{def:trace}
    For a graded $R$-module $M$, let $\tr_R(M)$ be the sum of the ideals $\phi(M)$ over all $\phi \in \Hom_R(M,R)$, i.e.
    \[
        \tr_R(M)=\sum_{\phi \in \Hom_R(M,R)}\phi(M).
    \]
    When there is no risk of confusion about the ring, we simply write $\tr(M)$.
\end{definition}

\begin{definition}[{\cite[Definition 2.2]{herzog2019trace}}]
    We say that $R$ is $\textit{nearly Gorenstein}$ if $\tr(\omega_R) \supseteq \mathbf{m}$.
    In particular, $R$ is Gorenstein
    if and only if $\tr(\omega_R) = R$.
\end{definition}

\begin{proposition}[{\cite[Lemma 1.1]{herzog2019trace}}]\label{prop:trace}
    Let $R$ be a ring and $I$ an ideal of $R$ containing
    a non-zero divisor of $R$. Let $Q(R)$ be the total
    quotient ring of fractions of $R$ and set $I^{-1} := \setcond{x \in Q(R)}{xI \subseteq R}.$
    Then
    \[
        \tr(I) = I \cdot I^{-1}.
    \]
\end{proposition}

\begin{definition}[{\cite[Chapter III, Proposition 3.2]{stanley2007combinatorics}}]\label{def:level}
    We say that $R$ is\;$\textit{level}$\: if all the degrees of the minimal generators of $\omega_R$ are the same.
\end{definition}

Let $R = \bigoplus_{n \ge 0} R_n$ and $S = \bigoplus_{n \ge 0} S_n$ be standard $\kk$-algebras and define their \textit{Segre product} $R\# S$ as the graded algebra:
\[
    R \# S = (R_0 \otimes_{\kk} S_0) \oplus (R_1 \otimes_{\kk} S_1) \oplus \cdots \subseteq R \otimes_{\kk}S.
\]
We denote a homogeneous element $x\otimes_{\kk}y \in R_i\otimes_{\kk}S_i$ by
$x\#y$.

\begin{proposition}[{\cite[Proposition 2.2 and Theorem 2.4]{herzog2019measuring}}]\label{ABC}
    Let $R_1, \cdots, R_s$ be standard graded Cohen-Macaulay toric $\kk$-algebras which have Krull dimension at least 2.
    Let $R = R_1 \# R_2 \# \cdots \# R_s$ be the Segre product.
    Then the following is true.
    $$\omega_R = \omega_{R_1} \# \omega_{R_2} \# \cdots \# \omega_{R_s} \quad \text{ and } \quad {\omega_R^{-1}} = {\omega_{R_1}^{-1}} \# {\omega_{R_2}^{-1}} \# \cdots \# {\omega_{R_s}^{-1}}.$$
\end{proposition}

\begin{lemma}\label{miyashita}
    Let $R_1,\ldots,R_s$ be homogeneous normal affine semigroup rings over infinite field $\kk$ which have Krull dimension at least 2.
    Let $R=R_1\#\cdots\# R_s$ be the Segre products.
    Then the following are true:
    \begin{itemize}
        \item[(1)] If $R$ is nearly Gorenstein, then $R_i$ is nearly Gorenstein for all $i$.
        \item[(2)] If $R_i$ is level for all $i$, then $R$ is level.
    \end{itemize}
\end{lemma}

\begin{proof}
    It suffices to prove the case $s=2$.
    Let $\bfx_1, \ldots, \bfx_n$ be $\kk$-basis of $(R_1)_1$ and $\bfy_1,\ldots, \bfy_m$ be a $\kk$-basis of$(R_2)_1$.

    (1): In this case, by using Proposition~\ref{ABC}, we get $\omega_{R} \cong \omega_{R_1}\#\omega_{R_2}$ and $\omega_{R}^{-1} \cong \omega_{R_1}^{-1}\#\omega_{R_2}^{-1}$.
    Then we may identify
    $\omega_{R}$ and ${\omega_R}^{-1}$
    with $\omega_{R_1}\#\omega_{R_2}$
    and $\omega_{R_1}^{-1}\#\omega_{R_2}^{-1}$,
    respectively.

    It is enough to show that $\bfx_i \in \tr(\omega_{R_1})$ for any $1 \leq i \leq n$.
    Since $R$ is nearly Gorenstein, there exist homogeneous elements $\vb_1\#\vb_2 \in \omega_{R_1}\#\omega_{R_2}$ and $\ub_1\#\ub_2 \in \omega_{R_1}^{-1}\#\omega_{R_2}^{-1}$ such that $\bfx_i \# \bfy_1=(\vb_1\#\vb_2)(\ub_1\#\ub_2) = (\vb_1\ub_1\#\vb_2\ub_2)$, by \cite[Proposition 4.2]{miyashita2022levelness}.
    Thus, we get $\bfx_i=\vb_1\ub_1 \in \tr(\omega_{R_1})$,
    so $R_1$ is nearly Gorenstein.
    In the same way as above, we can show that $R_2$ is also nearly Gorenstein.

    (2): First, $\omega_R \cong \omega_{R_1}\# \omega_{R_2}$ by Proposition~\ref{ABC}.
    Let $a_1$ and $a_2$ be the $a$-invariants of $R_1$ and $R_2$, respectively, and assume that $a_1 \leq a_2$.
        Since $R_1$ and $R_2$ are level,
        $\omega_{R_1} \cong \langle f_1,\cdots,f_r \rangle R_1$
        and $\omega_{R_2} \cong \langle g_1,\cdots,g_l \rangle R_2$
        where $\deg f_i=-a_1$ and
        $\deg g_j=-a_2$ for all $1 \leq i \leq r$, $1 \leq j \leq l$.
        Thus, since $\omega_R \cong \omega_{R_1}\# \omega_{R_2}$,
        we may identify
        $\omega_R$ with $\langle f_1,\cdots,f_r \rangle R_1\# \langle g_1,\cdots,g_l \rangle R_2$.
    We set
    \[
        V := \setcond{\bfy^{\bfb}g_j}{1 \leq j \leq l,\, \bfa \in \NN^m,\, \; \sum_{i=1}^m b_i = a_2 - a_1 },
    \]
    where
    $\bfy^{\bfa}:=\bfy_1^{a_1}\cdots\bfy_m^{a_m}$.
    Then $\omega_R = \langle f_i\#v\, : \, 1\leq i \leq r, v \in V \rangle R .$
    Therefore, $R$ is level.
\end{proof}

\subsection{Lattice polytopes and Ehrhart rings}\label{subsec_lattice}

We denote the natural pairing between an element $n \in (\RR^d)^*$ and an element $x \in \RR^d$ by $n(x)$.
Throughout this subsection, let $P \subset \RR^d$ be a lattice polytope, $\facets(P)$ be the set of facets of $P$, and $\verto(P)$ be the set of vertices of $P$.
Moreover, recall that we always assume $P$ is full-dimensional and has the facet presentation
\begin{align}\label{facet_pre}
    P = \setcond{x\in \RR^d}{n_F(x) \ge -h_F \text{ for all $F\in \facets(P)$}},
\end{align}
where each height $h_F$ is an integer and each inner normal vector $n_F \in (\ZZ^d)^*$ is a \emph{primitive} lattice point, i.e. a lattice point such that the greatest common divisor of its coordinates is $1$.

Let $C_P$ be the \textit{cone over $P$}, that is,
\[
    C_P = \RR_{\ge 0} (P \times \set{1}) = \setcond{(x,k) \in \RR^{d+1}}{n_F(x) \ge -k h_F \text{ for all $F\in \facets(P)$}}.
\]

We define the \textit{Ehrhart ring} of $P$ as
\[
    A(P) = \kk[C_P \cap \ZZ^{d+1}] = \kk[\tb^x s^k : k \in \NN \text{ and } x \in kP \cap \ZZ^d],
\]
where $\tb^x = t_1^{x_1} \cdots t_d^{x_d}$ and $x = (x_1, \ldots, x_d) \in kP \cap \ZZ^d$.
Note that the Ehrhart ring of $P$ is a normal affine semigroup ring, and hence it is Cohen-Macaulay.
Moreover, we can regard $A(P)$ as an $\NN$-graded $\kk$-algebra by setting $\deg(\tb^x s^k) = k$ for each $x \in kP \cap \ZZ^d$.

We also define another affine semigroup ring, the \textit{toric ring} of $P$, as
\[
    \kk[P] = \kk[\tb^x s : x \in P \cap \ZZ^d].
\]

\noindent The toric ring of $P$ is a standard $\NN$-graded $\kk$-algebra. 

It is known that $\kk[P]=A(P)$ if and only if $P$ has the integer decomposition property.
Here, we say that $P$ has the \emph{integer decomposition property} (i.e. $P$ is \emph{IDP}) if for all positive integers $k$ and all $x \in kP \cap \ZZ^d$, there exist $y_1, \ldots, y_k \in P \cap \ZZ^d$ such that $x = y_1 + \cdots + y_k$.

In order to describe the canonical module and the anti-canonical module of $A(P)$ in terms of $P$, we prepare some notation.

For a polytope or cone $K$, we denote the strict interior of $\sigma$ by $\into(\sigma)$.
Note that
\[
    \into(C_P) = \setcond{(x,k) \in \RR^{d+1}}{n_F(x) > -k h_F \text{ for all $F\in \facets(P)$}}.
\]
Moreover, we define
\[
    \ant(C_P) := \setcond{(x,k) \in \RR^{d+1}}{ n_F(x) \ge -k h_F -1 \text{ for all $F\in \facets(P)$}}.
\]
Then the following is true.

\begin{proposition}[see {\cite[Proposition 4.1 and Corollary 4.2]{herzog2019measuring}}]\label{prop:can_antican}
    The canonical module of $A(P)$ and the anti-canonical module of $A(P)$ are given by the following, respectively:
    \[
        \omega_{A(P)} = \left\langle \tb^x s^k : (x,k) \in \into(C_P) \cap \ZZ^{d+1} \right\rangle \text{ and }  \omega_{A(P)}^{-1} = \left\langle \tb^x s^k : (x,k) \in \ant(C_P) \cap \ZZ^{d+1} \right\rangle.
    \]
    Further, the negated $a$-invariant of $A(P)$ coincides with the codegree of $P$, i.e.
    \[
        a(A(P)) = -\min\setcond{k \in \ZZ_{\ge1}}{\strint(kP) \cap \ZZ^d \neq \varnothing}.
    \]
\end{proposition}

Let $A$ and $B$ be subsets of $\RR^d$.
Their \emph{Minkowski sum} is defined as
\[
    A + B \coloneqq \setcond{x + y}{x \in A, y \in B}.
\]

We recall that the \emph{(direct) product} of two polytopes $P \subset \RR^d$ and $Q \subset \RR^e$ is denoted by $P \times Q \subset \RR^{d+e}$.

Note that we can regard $P \times Q$ as the Minkowski sum of polytopes, as follows.
Let
\[
    P' = \setcond{(p, \underbrace{0, \ldots, 0}_{e})\in \RR^{d+e}}{p \in P} \text{ and } Q'= \setcond{(\underbrace{0, \ldots, 0}_{d}, q) \in \RR^{d+e}}{q \in Q}.
\]
Then, we can see that $P \times Q = P' + Q'$.
Conversely, suppose two polytopes $P', Q' \subset \RR^d$ satisfy the following condition:
for all $i \in [d] := \set{1, \ldots, d}$, we have that $\pi_i(P') =\set{0}$ or $\pi_i(Q') =\set{0}$, where $\pi_i : \RR^d \to \RR$ is the projection onto the $i$-th coordinate.
Then we can regard $P' + Q'$ as the product of two polytopes.
Moreover, let $P$ and $Q$ be two lattice polytopes.
It is known that $\kk[P \times Q]$ is isomorphic to the Segre product $\kk[P] \# \kk[Q]$.

Finally, we recall the definitions of (polar) duality and reflexivity of polytopes.

\begin{definition}
    Let $P \subset \RR^d$ be a polytope.
    Its \emph{(polar) dual} is
    \[
        P^* \coloneqq \setcond{n \in (\RR^d)^*}{n(x) \ge -1 \text{ for all } x \in P}.
    \]
    We call $P$ \emph{reflexive} if both $P$ and $P^*$ are lattice polytopes (with respect to the lattices $\ZZ^d$ and $(\ZZ^d)^*$, respectively).
\end{definition}

\bigskip

\section{Nearly Gorensteinness of lattice polytopes}\label{sec:nearly}
Throughout this section, the lattice polytope $P$ has the facet presentation (\ref{facet_pre}).
\begin{definition}
    We say that $P$ is \emph{Gorenstein} (resp. \emph{nearly Gorenstein}) if the Ehrhart ring $A(P)$ is Gorenstein (resp. nearly Gorenstein).
\end{definition}

There are well-known equivalent conditions of Gorensteinness in terms of the lattice polytope $P$ itself.
For instance, $P$ is Gorenstein if and only if there exists a positive integer $a$ such that a \emph{lattice translation} of $aP$ is \emph{reflexive}, i.e. $aP$ has a unique interior lattice point which has lattice distance $1$ to all facets of $aP$.


In this section, we will determine a necessary condition for $P$ to be nearly Gorenstein, in terms of the polytope $P$ itself.
This condition demands that $P$ has a particular Minkowski decomposition.
By taking a dual perspective, we see exactly the connection to reflexive polytopes.
Next, we will show that if $P$ satisfies the aforementioned necessary condition and is in some sense ``big enough'', then $P$ will be nearly Gorenstein.
We end the section by investigating the nearly Gorensteinness of Minkowski indecomposable lattice polytopes.
\bigskip

\subsection{Necessary conditions}\label{subsec:nec}
The main aim of this subsection is to show the first half of Theorem~\ref{thm:nec_suf_intro}.
Before we proceed, let us first introduce some helpful notation.
For a subset $X$ of $\RR^{d+1}$ and $k \in \ZZ$, let $X_k = \setcond{x \in \RR^d}{(x,k) \in X}$ be the $k$-th \emph{piece} of $X$.
Note the subtlety in our notation: while $X$ is a subset of $\RR^{d+1}$, its $k$-th piece $X_k$ is a subset of $\RR^d$.
Moreover, for a lattice polytope $P$, we denote its \emph{codegree} by $a_P$ -- see below Proposition~\ref{prop:can_antican} for the definition.
When it is clear from context, we simply write $a$ instead of $a_P$.

\begin{proposition}\label{prop:nearly}
    Let $P \subset \RR^d$ be a lattice polytope.
    Then $P$ is nearly Gorenstein if and only if
    \begin{equation}\label{eq:nG_cones}
        (C_P \cap \ZZ^{d+1}) \setminus \set{0} \subseteq \strint(C_P) \cap \ZZ^{d+1} + \ant(C_P) \cap \ZZ^{d+1}.
    \end{equation}
    In particular, if $P$ is nearly Gorenstein, then
    \begin{equation}\label{eq:nG_IDP}
        P \cap \ZZ^d = \strint(C_P)_a \cap \ZZ^d + \ant(C_P)_{1-a} \cap \ZZ^d.
    \end{equation}
    The converse also holds if $P$ is IDP.
\end{proposition}

\begin{proof}
    By definition, $P$ is nearly Gorenstein if and only if the trace $\trace(\omega)$ of the canonical ideal $\omega$ of $A(P)$ contains the maximal ideal $\mm$ of $A(P)$.
    By Proposition~\ref{prop:trace}, this trace is exactly the product $\omega_{A(P)} \cdot \omega_{A(P)}^{-1}$.
    Then, Proposition~\ref{prop:can_antican} tells us the monomial generators of $\omega$ and $\omega^{-1}$ in terms of the lattice points of $\strint(C_P)$ and $\ant(C_P)$.
    We finally note that the maximal ideal $\mm$ can be generated by the monomials $\tb^x s^k$, where $(x,k)$ are lattice points in $C_P \setminus \set{0}$.
    From this, it is clear to see that $P$ is nearly Gorenstein if and only if (\ref{eq:nG_cones}) holds.

    We next prove that (\ref{eq:nG_IDP}) follows from nearly Gorensteinness of $P$.
    First, note that the right hand side of (\ref{eq:nG_cones}) is contained in $C_P \cap \ZZ^{d+1}$ by definition.
    Therefore, when we take the $1$-st piece of all three sets, we obtain the equality
    \[
        P \cap \ZZ^d = (\strint(C_P) \cap \ZZ^{d+1} + \ant(C_P) \cap \ZZ^{d+1})_1.
    \]
    Note that when $P$ is Gorenstein, $\strint(C_P)_a \cap \ZZ^d$ and $\ant(C_P)_{-a} \cap \ZZ^d$ are singleton sets; therefore, the result easily follows.
    Otherwise, we claim that $\ant(C_P)_{1-b} \cap \ZZ^d$ is empty for all $b \ge a+1$.
    Since $\strint(C_P)_b$ is empty for $b < a$, we obtain the desired result.

    Finally, we show that the converse holds when $P$ is IDP.
    Let $(x,k) \in C_P \cap \ZZ^d \setminus \set{0}$.
    Since $P$ is IDP, there are $x_1, \ldots, x_k \in P \cap \ZZ^d$ such that $(x,k) = (x_1,1) + \cdots + (x_k,1)$.
    Further, each $x_i \in P \cap \ZZ^d$ can be written as the sum of lattice points in $\strint(C_P)$ and $\ant(C_P)$.
    Therefore, (\ref{eq:nG_cones}) holds and so $P$ is nearly Gorenstein.
\end{proof}

\bigskip

\begin{definition}\label{def:floor_rem}
    Let $P \subset \RR^d$ be a lattice polytope.
    We define its \emph{floor polytope} and \emph{remainder polytopes} as
    \[
        \floor{P} := \conv(\strint(P) \cap \ZZ^d) \qquad \text{ and } \qquad \set{P} := \conv(\ant(C_P)_{1-a} \cap \ZZ^d),
    \]
    respectively.
    Note that $\floor{P}$ coincides with $\conv(\strint(C_P)_1 \cap \ZZ^d)$.
\end{definition}

We collate a couple of easy facts about these polytopes and reformulate part of Proposition~\ref{prop:nearly} into the following statement.
\begin{lemma}\label{lem:flo_rem}
    Let $P \subset \RR^d$ be a lattice polytope. Then:
    \begin{enumerate}
        \item $\floor{aP} \subseteq \setcond{x \in \RR^d}{n_F(x) \ge 1 - ah_F \text{ for all $F\in \facets(P)$}}$;
        \item $\set{P} \subseteq \setcond{x \in \RR^d}{n_F(x) \ge (a-1)h_F - 1 \text{ for all $F \in \facets(P)$}}$;
        \item If $P$ is nearly Gorenstein, then $P \cap \ZZ^d = \floor{aP} \cap \ZZ^d + \set{P} \cap \ZZ^d$;
        \item If $P$ is IDP and $P \cap \ZZ^d = \floor{aP} \cap \ZZ^d + \set{P} \cap \ZZ^d$, then $P$ is nearly Gorenstein.
    \end{enumerate}
\end{lemma}

\begin{proof}
    Statements (1) and (2) follow immediately from the definition of the floor and remainder polytope.
    To prove statements (3) and (4), notice that the lattice points of $\strint(C_P)_a$ coincide with those of $\floor{aP}$ and the lattice points of $\ant(C_P)_{1-a}$ coincide with those of $\set{P}$.
    Then simply substitute this into Proposition~\ref{prop:nearly}.
\end{proof}

The following proposition is the first half of Theorem~\ref{thm:nec_suf_intro}:

\begin{proposition}\label{prop:dec}
    If $P$ is nearly Gorenstein, then $P =\floor{aP} + \set{P}$.
\end{proposition}
\begin{proof}
    Let $x \in \floor{aP}$ and $y \in \set{P}$.
    By statements (1) and (2) of Lemma~\ref{lem:flo_rem}, we have that, for all facets $F$ of $P$, $n_F(x + y) \ge 1 - ah_F + (a-1)h_F - 1 = -h_F$.
    So, $x+y \in P$.
    Therefore, we obtain that $\floor{aP} + \set{P} \subseteq P$.

    On the other hand, let $v$ be a vertex of $P$.
    Since $P$ is a lattice polytope, $v \in P \cap \ZZ^d$.
    Thus, by statement (3) of Lemma~\ref{lem:flo_rem}, can write $v$ as
    the sum of an element of $\floor{aP} \cap \ZZ^d$ and an element of $\set{P}\cap \ZZ^d$.
    This implies $P \subseteq \floor{aP} + \set{P}$.
\end{proof}

\begin{example}\label{stopsign}
    Consider the stop sign polytope, given by
    \[
        P = \conv\set{(1,0), (2,0), (3,1), (3,2), (2,3), (1,3), (0,2), (0,1)}.
    \]

    \begin{figure}[ht]
        \centering
        \begin{tikzpicture}[scale=0.8]
            \begin{axis}[
                x=1cm, y=1cm,
                axis lines=left,
                ymajorgrids=true,
                xmajorgrids=true,
                xmin=-1,
                xmax=4,
                ymin=-1,
                ymax=4,
                xtick={0,1,...,4},
                ytick={0,1,...,4},]
                \clip (-0.8, -0.8) rectangle (4.3, 4.3);
                \fill[fill=red, fill opacity=0.1] (1,0) -- (2,0) -- (3,1) -- (3,2) -- (2,3) -- (1,3) -- (0,2) -- (0,1) -- cycle;
                \draw [line width=1.5pt, color=black] (1,0)-- (2,0);
                \draw [line width=1.5pt, color=black] (2,0)-- (3,1);
                \draw [line width=1.5pt, color=black] (3,1)-- (3,2);
                \draw [line width=1.5pt, color=black] (3,2)-- (2,3);
                \draw [line width=1.5pt, color=black] (2,3)-- (1,3);
                \draw [line width=1.5pt, color=black] (1,3)-- (0,2);
                \draw [line width=1.5pt, color=black] (0,2)-- (0,1);
                \draw [line width=1.5pt, color=black] (0,1)-- (1,0);
            \end{axis}
        \end{tikzpicture}
        \begin{tikzpicture}[scale=0.8]
            \begin{axis}[
                x=1cm,y=1cm,
                axis lines=left,
                ymajorgrids=true,
                xmajorgrids=true,
                xmin=0,
                xmax=3,
                ymin=0,
                ymax=3,
                xtick={1,2,3},
                ytick={1,2,3},]
                \clip(0,0) rectangle (3,3);
                \fill[fill=red, fill opacity=0.1] (1,1) -- (2,1) -- (2,2) -- (1,2) -- cycle;
                \draw [line width=1.5pt, color=black] (1,1)-- (2,1);
                \draw [line width=1.5pt, color=black] (2,1)-- (2,2);
                \draw [line width=1.5pt, color=black] (2,2)-- (1,2);
                \draw [line width=1.5pt, color=black] (1,2)-- (1,1);
            \end{axis}
        \end{tikzpicture}
        \begin{tikzpicture}[scale=0.8]
            \begin{axis}[
                x=1cm,y=1cm,
                axis lines=left,
                ymajorgrids=true,
                xmajorgrids=true,
                xmin=-2,
                xmax=2,
                ymin=-2,
                ymax=2,
                xtick={-1,0,1,2},
                ytick={-1,0,1,2},]
                \clip(-2,-2) rectangle (2,2);
                \fill[fill=red,fill opacity=0.1] (1,0) -- (0,1) -- (-1,0) -- (0,-1) -- cycle;
                \draw [line width=1.5pt,color=black] (1,0)-- (0,1);
                \draw [line width=1.5pt,color=black] (0,1)-- (-1,0);
                \draw [line width=1.5pt,color=black] (-1,0)-- (0,-1);
                \draw [line width=1.5pt,color=black] (0,-1)-- (1,0);
            \end{axis}
        \end{tikzpicture}

        \caption{The stop sign polytope $P$ (left) with its floor polytope $\floor{P}$ (middle) and remainder polytope $\set{P}$ (right).}
        \label{fig:stop_sign}
    \end{figure}
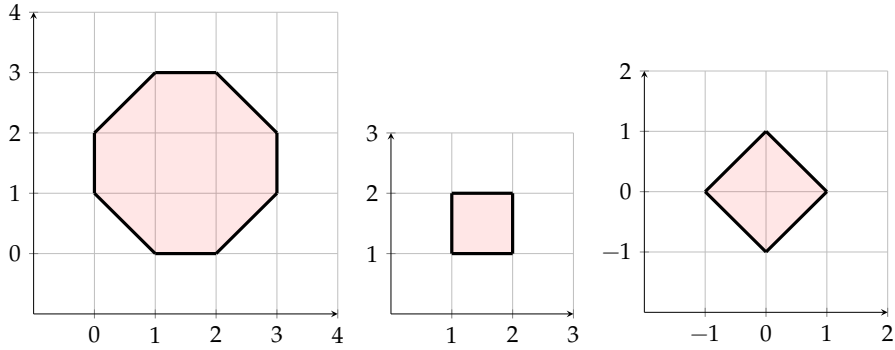

    First, we note that $a_P = 1$.
    Next, we may compute the floor and remainder polytopes:
    \[
        \floor{P} = \conv\set{(1,1), (2,1), (1,2), (2,2)} \quad \text{ and } \quad \set{P} = \conv\set{(1,0), (0,1), (-1,0), (0,-1)}.
    \]
    By taking the Minkowski sum of these polytopes, we see that $P$ satisfies the necessary condition to be Gorenstein given by Proposition~\ref{prop:dec}, i.e. $P = \floor{P} + \set{P}$.
    On the other hand, it is straightforward to verify that every lattice point of $P$ can be written as the sum of a lattice point of $\floor{P}$ and a lattice point of $\set{P}$.
    Since $P$ is IDP (as is true for all \emph{polygons}), statement (4) of Lemma~\ref{lem:flo_rem} informs us that $P$ is nearly Gorenstein.

    Finally, we remark that the remainder polytope $\set{P}$ is reflexive.
    This is not coincidence, as we will prove in Proposition~\ref{thm:conj1}.

\end{example}

\subsection{A sufficient condition}\label{subsec:suf}
In this subsection, we will explore sufficient conditions for a lattice polytope to be nearly Gorenstein; in particular, we will prove the second half of Theorem~\ref{thm:nec_suf_intro}.

We first note that the converse of Proposition~\ref{prop:dec} does not hold in general.

\begin{example}[compare {\cite[Example 1.1]{mustata_payne2005}}]
    Let $f = \frac13(e_1 + \cdots + e_6) \in \RR^6$, where $e_1, \ldots, e_6$ is a basis of the lattice $\ZZ^6$.
    Define a new lattice $L \coloneqq \ZZ^6 + f \cdot \ZZ$, and consider the lattice polytope
    \[
        Q \coloneqq \conv\set{e_1, \ldots, e_6, e_1 - f, \ldots, e_6 - f}
    \]
    with respect to the lattice $L$.
    Set $P \coloneqq 2Q$.
    Since $\floor{P} = \set{P} = Q$, it's easy to see that $P = \floor{P} + \set{P}$, meeting the necessary condition of Proposition~\ref{prop:dec} for nearly Gorensteinness.

    On the other hand, $Q$ is not IDP.
    In particular, $2Q \cap L \neq (Q \cap L) + (Q \cap L)$.
    Thus, $P = 2Q$ fails the necessary condition of statement (3) in Lemma~\ref{lem:flo_rem}, and so $P$ is not nearly Gorenstein.
\end{example}

So, we need to make more assumptions about $P$ in order to be guaranteed nearly Gorensteinness.
This brings us to the following result, which is the second half of Theorem~\ref{thm:nec_suf_intro}:

\begin{theorem}\label{thm:big_enough_nG}
    Let $P \subset \RR^d$ be a lattice polytope satisfying $P = \floor{aP} + \set{P}$.
    Then there exists some integer $K \ge 1$ (depending on $P$) such that for all $k \ge K$, the polytope $kP$ is nearly Gorenstein.
\end{theorem}

In order to prove the above, we rely on a few key ingredients.
The first ingredient is an extension of known results from the reflexive case, which appear in \cite{hibi1992}.

\begin{lemma}\label{lem:kP}
    Let $P \subset \RR^d$ be a lattice polytope satisfying $P = \floor{aP} + \set{P}$.
    Then the following statements hold:
    \begin{enumerate}
        \item $kP = \floor{(k+a-1)P} + \set{P}$, for all $k \ge 1$;
        \item $\floor{k'P} = \floor{aP} + (k'-a)P$, for all $k' \ge a$.
    \end{enumerate}
\end{lemma}

Before we give the proof, we will restrict these statements to the reflexive case for the sake of comparison.
First, we have $a=1$.
Next, since $\floor{P}$ is the origin, $P = \set{P}$.
So, for reflexive polytopes, the statement (1) is equivalent to $kP = \floor{kP} + P$.
After cancellation by $P$, we obtain the reflexive version of statement (2): $\floor{kP} = (k-1)P$.

\begin{proof}[Proof of Lemma~\ref{lem:kP}]
    Let $k \ge 1$ be an integer.
    Throughout this proof, we repeatedly use the two inequalities appearing in statements (1) and (2) of Lemma~\ref{lem:flo_rem}.
    We also use the inequalities appearing in the facet presentations for $P$ and its dilates.

    We first prove the ``$\supseteq$'' part of statement (1), i.e. that
    \begin{equation}\label{eq:kP_decomp}
        kP \supseteq \floor{(k+a-1)P} + \set{P}, \text{ for all } k \ge 1.
    \end{equation}
    Let $x \in \floor{(k+a-1)P}$ and $y \in \set{P}$.
    Then $n_F(x+y) \ge (1 - (k+a-1)h_F) + ((a-1)h_F - 1) = -kh_F$, for all facets $F$ of $P$.
    Thus, $x+y \in kP$.

    Next, we note that $kP = (k-1)P + \floor{aP} + \set{P}$.
    We substitute this into (\ref{eq:kP_decomp}), then cancel $\set{P}$ from both sides to obtain $\floor{(k+a-1)P} \subseteq (k-1)P + \floor{aP}$.

    We now prove the reverse inclusion of the above.
    Let $x \in (k-1)P$ and $y \in \floor{aP}$.
    Then, $n_F(x+y) \geq -(k-1)h_F + (1-ah_F) = 1 - (k+a-1)h_F$.
    Therefore, $x+y \in \floor{(k+a-1)P}$.
    Thus, we obtain the equality $\floor{(k+a-1)P} = (k-1)P + \floor{aP}$.
    Setting $k' \coloneqq k + a - 1$ then gives us statement (2).
    Adding $\set{P}$ to both sides gives us statement (1).
\end{proof}

The main ingredient in proving Theorem~\ref{thm:big_enough_nG} is a result of Haase and Hofmann, which allows us to guarantee that the second condition of statement (4) of Lemma~\ref{lem:flo_rem} holds.

\begin{theorem}[{\cite[Theorem 4.2]{haase_hofmann_2017}}]\label{thm:idp_pair}
    Let $P, Q \subset \RR^d$ be rational polytopes such that the normal fan $\Nc(P)$ of $P$ is a refinement of the normal fan $\Nc(Q)$ of $Q$.
    Suppose also that for each edge $E$ of $P$, the corresponding face $E'$ of $Q$ has lattice length $\ell_{E'}$ satisfying $\ell_E \ge d \ell_{E'}$.
    Then $(P+Q) \cap \ZZ^d = (P \cap \ZZ^d) + (Q \cap \ZZ^d)$.
\end{theorem}

In order to guarantee the first condition of statement (4) of Lemma~\ref{lem:flo_rem}, we need this next result:

\begin{theorem}[{\cite[Theorem~1.3.3]{viet1997normal}}]\label{cor:idp_d-1}
    Let $P \subset \RR^d$ be a lattice polytope.
    Then $(d-1)P$ is IDP.
\end{theorem}

We are now ready to give the proof.

\begin{proof}[Proof of Theorem~\ref{thm:big_enough_nG}]
    We first wish to find a suitable $K$ which satisfies
    \[
        kP \cap \ZZ^d = \floor{kP} \cap \ZZ^d + \set{kP} \cap \ZZ^d, \text{ for all } k \ge K.
    \]
    Looking at statement (2) of Lemma~\ref{lem:kP}, we see that $(k-a)P$ is a Minkowski summand of $\floor{kP}$; thus, we get a crude lower bound on the length of the edges of $\floor{kP}$: for $k \ge a$, every edge $E$ of $\floor{kP}$ has lattice length $\ell_E \ge k-a$.
    Denote by $L$ the maximum edge length of $\set{aP}$ and set $K \coloneqq dL + a$.
    Note that for $k \ge a$, the polytopes $\set{kP}$ and $\set{aP}$ coincide.
    So, for all $k \ge K$, every edge $E$ of $\floor{kP}$ will have lattice length $\ell_E \ge k-a \ge dL$.

    Further, statement (2) of Lemma~\ref{lem:kP} implies that, for $k \ge a+1$, the normal fan $\Nc(\floor{kP})$ coincides with $\Nc(P)$.
    Hence, $\Nc(\floor{kP})$ is a refinement of the normal fan of $\set{kP}$.
    Thus, we may apply Theorem~\ref{thm:idp_pair}, obtaining that $kP \cap \ZZ^d = \floor{kP} \cap \ZZ^d + \set{kP} \cap \ZZ^d$.

    Finally, since $a,L \ge 1$, we see that $K \ge d-1$.
    Thus, by Theorem~\ref{cor:idp_d-1}, we have that $kP$ is IDP.
    Therefore, by statement (4) of Lemma~\ref{lem:flo_rem}, we can conclude that $kP$ is nearly Gorenstein for all $k \ge K$.
\end{proof}

\begin{remark}
    We say that a graded ring $R$ is \emph{Gorenstein on the punctured spectrum} \cite{herzog2019trace} if $\trace(\omega_R)$ contains $\mm^k$ for some integer $k \ge 0$.
    If $k=0$, this is just the Gorenstein condition; if $k=1$, it is the nearly Gorenstein condition.
    Now, for a lattice polytope $P \subset \RR^d$, it can be shown that its Ehrhart ring $A(P)$ is Gorenstein on the punctured spectrum if there exists a positive integer $K$ such that $kP \cap \ZZ^d$ coincides with $(\strint(C_P) \cap \ZZ^{d+1}+\ant(C_P)\cap\ZZ^{d+1})_k$, for all $k \ge K$.
    Therefore, using Theorem~\ref{thm:big_enough_nG}, it's straightforward to show that all lattice polytopes $P$ satisfying $P = \floor{aP} + \set{P}$ are Gorenstein on the punctured spectrum.
\end{remark}



\subsection{Decompositions of nearly Gorenstein polytopes}\label{deco}
In this subsection, we first prove Theorem~\ref{thm:conj1_intro}.
This naturally leads to an investigation of whether nearly Gorenstein polytopes decompose into the Minkowski sum of Gorenstein polytopes (Questions~\ref{ques:1} and~\ref{ques:2}).
We prove Theorem~\ref{thm:normal_fan_intro}, which leads to a way to systematically construct examples of nearly Gorenstein polytopes.
This is then used to find a counterexample to Questions~\ref{ques:1} and~\ref{ques:2}.
Finally, we conclude the section with a result about indecomposable nearly Gorenstein polytopes.

\begin{theorem}[{Theorem~\ref{thm:conj1_intro}}]\label{thm:conj1}
    Let $P \subset \RR^d$ be a lattice polytope which satisfies $P = \floor{aP} + \set{P}$.
    Then we have
    \[
        \floor{aP} = \setcond{x \in \RR^d}{n_F(x) \ge 1 - ah_F \text{ for all } F \in \facets(P)} \text{\;and }
    \]
    \[
        \set{P} = \setcond{x \in \RR^d}{n_F(x) \ge (a-1)h_F -1 \text{ for all $F\in \facets(P)$}}.
    \]
    In particular, the right hand sides of the equalities are lattice polytopes.
    Furthermore, if $a = 1$, then $\set{P}$ is a reflexive polytope.
\end{theorem}

\begin{proof}
    Label the two polytopes on the right-hand sides as $Q_1$ and $Q_2$, respectively.
    It's straightforward to see that $\floor{aP} = \conv(Q_1 \cap \ZZ^d)$ and $\set{P} = \conv(Q_2 \cap \ZZ^d)$.
    Thus, $\floor{aP} \subseteq Q_1$ and $\set{P} \subseteq Q_2$.
    Ultimately, we want to prove the reverse inclusions but first, we must show an intermediate equality: $P = Q_1 + Q_2$.
    Let $x \in Q_1$ and $y \in Q_2$.
    Then, for all facets $F$ of $P$, we have $n_F(x + y) \ge 1 - ah_F + (a-1)h_F - 1 = -h_F$.
    Thus, $x + y \in P$ and so, $Q_1 + Q_2 \subseteq P$.
    Conversely, if we combine this with our assumption that $P = \floor{aP} + \set{P}$, we obtain that, in fact, $P = Q_1 + Q_2$.

    We now use the above equality to obtain that $\floor{aP} = Q_1$ and $\set{P} = Q_2$, as follows.
    Assume towards a contradiction that $Q_1 \not\subseteq \floor{aP}$, i.e. there exists a vertex $v$ of $Q_1$ which doesn't belong to $\floor{aP}$.
    Choose a normal vector $n \in (\RR^d)^*$ which achieves its minimal value $h_1$ over $Q_1$ \emph{only} at $v$ (i.e. $n$ lies in the interior of the cone $\sigma_v$ in the (inner) normal fan $\Nc(Q_1)$ which corresponds to $v$).
    Denote by $h_2$ the minimal evaluation of $n$ over $Q_2$
    Then, the minimal evaluation of $n$ over $P$ is $h_1 + h_2$.
    However, for all $x \in \floor{aP}$ and $y \in \set{P}$, we have that $n(x+y) > h_1 + h_2$.
    This contradicts the fact that $P = \floor{aP} + \set{P}$.
    Therefore, the vertices of $Q_1$ coincide with the vertices of $\floor{aP}$; in particular, $\floor{aP} = Q_1$.
    We similarly obtain that $\set{P} = Q_2$.

    Next, since $\floor{aP}$ and $\set{P}$ are lattice polytopes by definition, we note that $Q_1$ and $Q_2$ are lattice polytopes in this situation.

    Finally, suppose we are in the case when $P$ has an interior lattice point, i.e. $a=1$.
    By substituting this into the second equality, we see that the remainder polytope $\set{P}$ is indeed reflexive as all its facets lie at height $1$.

\end{proof}

In contrast, when $P$ has no interior points, the remainder polytope $\set{P}$ is not necessarily even Gorenstein.

\begin{example}
    Consider the polytope
    \[
        P = \conv\set{(0,0,0), (2,0,0), (1,1,0), (0,1,0), (0,0,1), (2,0,1), (1,1,1), (0,1,1)}.
    \]
    We can verify that $P$ is nearly Gorenstein and IDP, but the remainder polytope $\set{P}$ is not Gorenstein.
    However, $\set{P}$ can be written as the Minkowski sum of
    \[
        \conv\set{(0,0,0), (1,0,0), (0,1,0)} \quad \text{ and } \quad \conv\set{(-1,-1,-1), (-1,-1,0)},
    \]
    which are both Gorenstein.
\end{example}

We see similar behavior when studying the nearly Gorensteinness for certain restricted classes of polytopes.
This motivated us to pose the following question.

\begin{question}\label{ques:1}
    If $P$ is nearly Gorenstein, then can we write $P = P_1 + \cdots + P_s$ for some Gorenstein lattice polytopes $P_1, \ldots, P_s$?
\end{question}

We recall that $P$ is (\textit{Minkowski}) \textit{indecomposable} if $P$ is not a singleton and if there exist lattice polytopes $P_1$ and $P_2$ with $P=P_1+P_2$,
then either $P_1$ or $P_2$ is a singleton.
Note that if $P$ is not a singleton, then we can write $P=P_1+\cdots +P_s$ for some indecomposable lattice polytopes $P_1,\ldots, P_s$.

Then, Question~\ref{ques:1} can be rephrased as:
\begin{question}\label{ques:2}
    If $P$ has an indecomposable non-Gorenstein lattice polytope as a Minkowski summand,
    then is $P$ not nearly Gorenstein?
\end{question}
This question has a positive answer for IDP $(0,1)$-polytopes, which is shown in Section~\ref{sec:(0,1)}.
For the remainder of this section, we will build up some machinery which allows for the efficient construction of nearly Gorenstein polytopes.
We then use this in Example~\ref{exa:counter} to give an answer to Questions~\ref{ques:1}~and~\ref{ques:2}.

\begin{theorem}[{Theorem~\ref{thm:normal_fan_intro}}]\label{thm:normal_fan}
    Let $P \subset \RR^d$ be a nearly Gorenstein polytope.
    Then there exists a reflexive polytope $Q \subset \RR^d$ such that
    \[
        P = \setcond{x \in \RR^d}{n(u) \ge -h_n \text{ for all } n \in \boundary Q^* \cap (\ZZ^d)^*},
    \]
    where $h_n$ are integers.
    Moreover, the inequalities defined by $n \in \verto(Q^*)$ are irredundant.
    Furthermore, the number of facets of a nearly Gorenstein polytope is bounded by a constant depending on the dimension $d$.
\end{theorem}

Before we dive into the proof, it will be useful to have the following lemma.

\begin{lemma}\label{lemma:aP_decomp}
    Let $P$ be a lattice polytope satisfying $P = \floor{aP} + \set{P}$.
    Then $aP = \floor{aP} + \set{aP}$.
    Moreover, $\set{aP} = (a-1)P + \set{P}$.
\end{lemma}

\begin{proof}
    We first wish to show that $(a-1)P + \set{P} \subseteq \set{aP}$.
    Let $x \in (a-1)P$ and $y \in \set{P}$.
    Then, by Lemma~\ref{lem:flo_rem}~(2), $n_F(x+y) \ge -(a-1)h_F + (a-1)h_F - 1 = -1$, for all facets $F$ of $P$.
    So, $x+y \in \set{aP}$.
    Thus, $(a-1)P + \set{P} \subseteq \set{aP}$.

    We can add $\floor{aP}$ to both sides of the inclusion to get $aP \subseteq \floor{aP} + \set{aP}$.

    We next wish to show the reverse inclusion of the above.
    Let $z \in \floor{aP}$ and $w \in \set{aP}$.
    Then $n_F(z+w) \ge (1-ah_F) - 1 = -ah_F$, for all facets $F$ of $P$.
    So, $z+w \in aP$.
    Therefore, $\floor{aP} + \set{aP} \subseteq aP$.
    Combining the two inclusions gives the desired equality: $aP = \floor{aP} + \set{aP}$.

    Moreover, we obtain that $\floor{aP} + \set{P} + (a-1)P = \floor{aP} + \set{aP}$.
    Since Minkowski addition of convex sets satisfies the cancellation law, we may cancel both sides by $\floor{aP}$ to obtain the equality $\set{aP} = (a-1)P + \set{P}$.
\end{proof}

\begin{proof}[Proof of Theorem~\ref{thm:normal_fan}]
    We wish to study the (inner) normal fan $\Nc(P)$ of $P$, as it's enough to show that its primitive ray generators all lie in $\boundary Q^* \cap (\ZZ^d)^*$, for some reflexive polytope $Q \subset \RR^d$.
    Since dilation has no effect on the normal fan, we may pass to the normal fan of $aP$.
    Now, by Lemma~\ref{lemma:aP_decomp}, $aP$ has a Minkowski decomposition into $\floor{aP}$ and $\set{aP}$.
    Thus, $\Nc(aP)$ is the common refinement of $\Nc(\floor{aP})$ and $\Nc(\set{aP})$.
    By Proposition~\ref{thm:conj1}, we obtain that $Q \coloneqq \set{aP}$ is a reflexive polytope.
    Hence, the primitive ray generators of $\Nc(Q)$ are vertices of the reflexive polytope $Q^* \subset (\RR^d)^*$; in particular, they are lattice points lying in the boundary of $Q^*$.

    We next look at the contribution to $\Nc(aP)$ coming from $\floor{aP}$.
    Let $n \in (\ZZ^d)^*$ be a primitive ray generator of $\Nc(\floor{aP})$.
    Then, by definition of the remainder polytope, $n(x) \ge -1$, for all $x \in Q$.
    But now, this means that $n$ lies in $Q^*$.
    So, since $n \neq 0$ and $Q$ is reflexive, we obtain that $n \in \boundary Q^* \cap (\ZZ^d)^*$.
    Therefore, we have now shown that the primitive ray generators of $\Nc(P) = \Nc(aP)$ contain the vertices of $Q^*$, and that they all lie in $\boundary Q^* \cap (\ZZ^d)^*$.

    Finally, we note that the number of facets of a nearly Gorenstein polytope $P \subset \RR^d$ is bounded by $c_d \coloneqq \sup_Q |\boundary Q^* \cap (\ZZ^d)^*|$, where $Q$ runs over all $d$-dimensional reflexive polytopes.
    Since there are only finitely reflexive polytopes in each dimension $d$, and all polytopes only have a finite number of boundary points, we see that $c_d$ is a finite number.
\end{proof}

We will now detail how to construct nearly Gorenstein polytopes.
First, choose a reflexive polytope $Q \subset \RR^d$.
Then, choose a (possibly empty) subset $S'$ of the boundary lattice points of $Q^*$ which are not vertices of $Q^*$.
Now, for each $n \in S \coloneqq S' \cup \verto(Q^*)$, choose the height $h_n \in \ZZ$.
Construct a polytope $P'$ defined by $n(x) \ge -h_n$ for all $n \in S$, and assert that none of these inequalities are redundant.
Next, we can dilate $P'$ to $rP'$ so that it's a lattice polytope which contains an interior lattice point.
By construction, its remainder polytope $\set{rP'}$ coincides with the reflexive polytope $Q$.
In practice, $rP'$ has a Minkowski decomposition into $\floor{rP'}$ and $\set{rP'}$, but we don't yet have a proof that this always holds.
Finally, we can use Theorem~\ref{thm:big_enough_nG} to dilate $rP'$ even further to $P \coloneqq krP'$ so that $P = \floor{P} + \set{P}$ is nearly Gorenstein.

\begin{example}\label{exa:counter}
    Consider the polytope
    \[
        P = \conv\set{(-4, -3, -4),
        (-3, -1, -3),
        (-2, -2, -3),
        (0, 1, 4),
        (0, 4, 1),
        (3, 1, 1)}.
    \]
    Note that $P$ has many interior lattice points, so we are definitely in the case $a_P = 1$.
    Its floor polytope is
    \[
        \floor{P} = \conv\set{(-3,-2,-3), (0,3,1), (0,1,3), (2,1,1)}.
    \]
    This is an indecomposable simplex, which is not Gorenstein.
    Its remainder polytope is
    \[
        \set{P} = \conv\set{(-1,-1,-1), (1,0,0), (0,1,0), (0,0,1)},
    \]
    which is clearly reflexive.
    We have $P = \floor{P} + \set{P}$.
    We use Magma \cite{magma} to verify that $P \cap \ZZ^3 = (\floor{P} \cap \ZZ^3) + (\set{P} \cap \ZZ^3)$ and that $P$ is IDP.
    Thus, we may conclude by Lemma~\ref{lem:flo_rem} that $P$ is a nearly Gorenstein polytope.
    It has no Minkowski decomposition into Gorenstein polytopes; therefore, we may answer Questions~\ref{ques:1} and~\ref{ques:2} in the negative.

\end{example}

\bigskip




We end this section by giving the following theorem about nearly Gorensteinness of indecomposable polytopes, which plays an important role in the characterisation of nearly Gorenstein $(0,1)$-polytopes in Section~\ref{sec:(0,1)}.

\begin{theorem}\label{thm:indecomposable}
    Let $P$ be an indecomposable lattice polytope.
    Then, $P$ is nearly Gorenstein if and only if $P$ is Gorenstein.
\end{theorem}



\begin{proof}
    It is already clear that Gorensteinness implies nearly Gorensteinness, so we just have to treat the converse implication.
    Suppose that $P$ is nearly Gorenstein.
    By Proposition~\ref{prop:dec}, we have that $P = \floor{aP} + \set{P}$.
    Since $P$ is indecomposable, either (i) $\floor{aP}$ is a singleton or (ii) $\set{P}$ is a singleton.

    We first deal with case (i).
    Consider $aP$.
    By Lemma~\ref{lemma:aP_decomp}, $aP = \floor{aP} + \set{aP}$.
    Thus, $aP$ is a translation of $\set{aP}$.
    By Proposition~\ref{thm:conj1}, $\set{aP}$ is reflexive.
    Thus, $P$ is Gorenstein.

    The argument for case (ii) is similar.
    We consider $\set{aP}$.
    By Lemma~\ref{lemma:aP_decomp}, $\set{aP} = (a-1)P + \set{P}$.
    Proposition~\ref{thm:conj1} tells us that $\set{aP}$ is reflexive; therefore, $(a-1)P$ is a translation of a reflexive polytope.
    But this is an absurdity as it implies that $(a-1)P$ has an interior lattice point, contradicting that the codegree of $P$ is $a$.
    Thus, this case cannot occur.

\end{proof}

\bigskip

\section{Nearly Gorenstein $(0,1)$-polytopes}\label{sec:(0,1)}

In this section, we consider the case of $(0,1)$-polytopes.
We provide the characterisation of nearly Gorenstein $(0,1)$-polytopes which are IDP.
Moreover, we also characterise nearly Gorenstein edge polytopes of graphs satisfying the odd cycle condition and characterise nearly Gorenstein graphic matroid polytopes.

\subsection{The characterisation of nearly Gorenstein $(0,1)$-polytopes}\label{subsec:char}

\begin{lemma}\label{dec01}
    Let $P\subset \RR^d$ be a $(0,1)$-polytope.
    Then, after a change of coordinates, we can write $P=P_1\times \cdots \times P_s$ for some indecomposable $(0,1)$-polytopes $P_1,\ldots,P_s$.
\end{lemma}

\begin{proof}
    As mentioned in Section~\ref{sec:nearly}, we can write $P=P_1'+\cdots+P_s'$ for some indecomposable lattice polytopes $P_1',\ldots,P_s'$.

    First, we show that we can choose $P_1',\ldots,P_s'$ so that these are $(0,1)$-polytopes.
    Suppose that we can write $P=P_1'+P_2'$ for some lattice polytopes $P_1'$ and $P_2'$.
    Then, for any $v\in P_1'\cap \ZZ^d$ and for any $u\in P_2'\cap \ZZ^d$, $v+u$ is a $(0,1)$-vector.
    Therefore, for any $i\in [d]$, $\pi_i(P_1'\cap \ZZ^d)$ can take one of the following forms: (i) $\{w_i\}$ or (ii) $\{w_i,w_i+1\}$ for some $w_i\in \ZZ$.
    In case (i), $\pi_i(P_2'\cap \ZZ^d)$ is equal to $\{-w_i\}$, $\{-w_i+1\}$ or $\{-w_i,-w_i+1\}$.
    In case (ii), $\pi_i(P_2'\cap \ZZ^d)$ is equal to $\{-w_i\}$.
    Thus, in all cases, $P_1'-w$ and $P_2'+w$ are $(0,1)$-polytopes and we have $P=(P_1'-w)+(P_2'+w)$, where $w=(w_1,\ldots,w_d)$.

    Moreover, if we can write $P=P_1'+P_2'$ for some $(0,1)$-polytopes $P_1'$ and $P_2'$, then we can see that either $\pi_i(P_1')$ or $\pi_i(P_2')$ is equal to $\{0\}$ for any $i\in [d]$.
    Therefore, after a change of coordinates, we can write $P=P_1\times P_2$ for some $(0,1)$-polytopes $P_1$ and $P_2$.
\end{proof}

Now, we provide the main theorem of this section.

\begin{theorem}\label{01}
    Let $P$ be an IDP $(0,1)$-polytope.
    Then,
    $P$ is nearly Gorenstein
    if and only if you can write $P=P_1\times \cdots \times P_s$ for some Gorenstein
    $(0,1)$-polytopes $P_1,\ldots,P_s$ with $|a_{P_i}-a_{P_j}|\le 1$ for $1\le i<j\le s$.
\end{theorem}

\begin{proof}
    It follows from Lemma~\ref{dec01} that we can write $P=P_1\times \cdots \times P_s$ for some indecomposable $(0,1)$-polytopes $P_1,\ldots, P_s$.
    Thus, we have $\kk[P]\cong \kk[P_1]\#\cdots \#\kk[P_s]$.
    Note that if $P$ is IDP, then so is $P_i$ for each $i\in [s]$, and $A(P)$ (resp. $A(P_i)$) coincides with $\kk[P]$ (resp. $\kk[P_i]$).
    Therefore, since $P$ is nearly Gorenstein, $\kk[P]$ is nearly Gorenstein, and hence $\kk[P_i]$ is also nearly Gorenstein from Lemma~\ref{miyashita} (1).
    Furthermore, $P_i$ is nearly Gorenstein.
    Since $P_i$ is indecomposable, $P_i$ is Gorenstein by Theorem~\ref{thm:indecomposable}.
    Moreover, it follows from \cite[Corollary 2.8]{herzog2019measuring} that $|a_{P_i}-a_{P_j}|\le 1$ for $1\le i<j\le s$.

    The converse also holds from \cite[Corollary 2.8]{herzog2019measuring}.
\end{proof}

From this theorem, we immediately obtain the following corollaries:

\begin{corollary}
    Question~\ref{ques:1} is true for IDP $(0,1)$-polytopes.
\end{corollary}

\begin{corollary}\label{01NGimpliesL}
    Let $P$ be an IDP $(0,1)$-polytope.
    If $\kk[P]$ is nearly Gorenstein, then $\kk[P]$ is level.
\end{corollary}
\begin{proof}
    It follows immediately from Lemma~\ref{miyashita} (2) and Theorem~\ref{01}.
\end{proof}

The result of Theorem~\ref{01} can be applied to many classes of $(0,1)$-polytopes such as order polytopes and stable set polytopes.

Order polytopes, which were introduced by Stanley \cite{stanley1986twoposet}, arise from posets.
Let $\Pi$ be a poset equipped with a partial order $\preceq$.
The Ehrhart ring of the order polytope of a poset $\Pi$ is called the Hibi ring of $\Pi$, denoted by $\kk[\Pi]$.
It is known that Hibi rings are standard graded (\cite{hibi1987distributiv}).
For a subset $I \subset P$, we say that $I$ is a \textit{poset ideal} of $P$ if $p \in I$ and $q \preceq p$ then $q \in I$.
According to \cite{stanley1986twoposet}, the characteristic vectors of poset ideals in $\RR^\Pi$ are precisely the vertices of the order polytope of $\Pi$ (hence order polytopes are $(0,1)$-polytopes).
By this fact, we can see that the order polytope of a poset $\Pi$ is indecomposable if and only if $\Pi$ is connected.
Nearly Gorensteinness of Hibi rings have been studied in \cite{herzog2019trace}.
 It is shown that $\kk[\Pi]$ is nearly Gorenstein if and only if $\Pi$ is the disjoint union of pure connected posets $\Pi_1,\ldots,\Pi_q$ such that their ranks of any two also can only differ by at most 1.
 Moreover, in this case, $\kk[\Pi_i]$ is Gorenstein and $\kk[\Pi] \cong \kk[\Pi_1]\#\cdots \#\kk[\Pi_s]$.
 Therefore, its characterisation can be derived from Theorem~\ref{01}.

 Stable set polytopes, which were introduced by Chv\'{a}tal \cite{chvatal1975stable}, arise from graphs.
 For a finite simple graph $G$ on the vertex set $V(G)$ with the edge set $E(G)$, the stable set polytope of $G$, denoted by $\Stab_G$, is defined as the convex hull of the characteristic vectors of stable sets of $G$ in $\RR^{V(G)}$, hence $\Stab_G$ is a $(0,1)$-polytope.
 Here, we say that a subset $S$ of $V(G)$ is a \textit{stable set} if $\set{v,u}\notin E(G)$ for any $v,u\in S$.
This implies that $\Stab_G$ is indecomposable if and only if $G$ is connected.
Stable set polytopes behave well for perfect graphs.
For example, $\Stab_G$ is IDP if $G$ is perfect (cf.\cite{ohsugi2001compressed}).
 Moreover, the characterisation of nearly Gorenstein stable set polytopes of perfect graphs has been given in \cite{hibi2021nearly,miyazaki2022h-perfect}.
 Let $G$ be a perfect graph with connected components $G_1,\ldots,G_s$ and let $\delta_i$ denote the maximal cardinality of cliques of $G_i$.
 Then, it is known that $\Stab_G$ is nearly Gorenstein if and only if the maximal cliques of each $G_i$ have the same cardinality and $|\delta_i-\delta_j| \le 1$ for $1 \le i < j \le s$.
In this case, as in the case of order polytopes, $\kk[\Stab_{G_i}]$ is Gorenstein and $\kk[\Stab_G]\cong \kk[\Stab_{G_1}]\#\cdots \#\kk[\Stab_{G_s}]$.
 Therefore, its characterisation can also follow from Theorem~\ref{01}.

 Furthermore, by using this theorem, we can study the nearly Gorensteinness of other classes of $(0,1)$-polytopes.

\subsection{Nearly Gorenstein edge polytopes}
First, we define the edge polytope and edge ring of a graph.
We refer the reader to \cite[Section 5]{herzog2018binomial} and \cite[Chapters 10 and 11]{villarreal2001monomial} for an introduction to edge rings.

Let $G$ be a finite simple graph on the vertex set $V(G) = \set{1,\ldots,d}$ with the edge set $E(G)$.
Given an edge $e = \set{i,j} \in E(G)$, let $\rho(e) := \eb_i + \eb_j$, where $\eb_i$ denotes the $i$-th unit vector of $\RR^d$ for $i \in [d]$.
We define the \emph{edge polytope} $P_G$ of $G$ as follows:
\begin{align*}
    P_G = \conv\setcond{\rho(e)}{e \in E(G)}.
\end{align*}
The toric ring of $P_G$ is called the \emph{edge ring} of $G$, denoted by $\kk[G]$ instead of $\kk[P_G]$.

Let $G_1,\ldots,G_s$ be the connected components of $G$.
From the definition of edge polytope, we can see that $\kk[G]\cong \kk[G_1]\otimes \cdots \otimes \kk[G_s]$.
Therefore, in considering the characterisation of nearly Gorenstein edge polytopes, we may assume that $G$ is connected.

Moreover, for a connected graph $G$, $P_G$ is IDP if and only if $G$ satisfies the \textit{odd cycle condition},
in other words,
for each pair of odd cycles $C$ and $C'$ with no common vertex, there is an edge $\set{v,v'}$ with $v \in V(C)$ and $v' \in V(C')$
(see \cite{ohsugi1998normal, simis1994ideal}).


Gorenstein edge polytopes have been investigated in \cite{ohsugi2006GorEdge}.
We now state the characterisation of nearly Gorenstein edge polytopes.

\begin{corollary}\label{cor:edge_polytopes}
    Let $G$ be a connected simple graph satisfying the odd cycle condition.
    Then, the edge polytope $P_G$ of $G$ is nearly Gorenstein if and only if $P_G$ is Gorenstein or $G$ is the complete bipartite graph $K_{n,n+1}$ for some $n\ge 2$.
\end{corollary}

\begin{proof}
    If $P_G$ is nearly Gorenstein, then Theorem~\ref{01} allows us to write $P_G=P_1\times \cdots \times P_s$ for some indecomposable Gorenstein $(0,1)$-polytopes $P_1,\ldots,P_s$.
    Then, we have $s\le 2$ since $P_G\subset \{(x_1,\ldots,x_d)\in \RR^d : x_1+\cdots +x_d =2\}$, where $d=|V(G)|$.
    In the case $s=1$, $P_G$ is Gorenstein.
    If $s=2$, we can see that $P_1=\conv\{\eb_1,\ldots,\eb_n\}\subset \RR^n$ and $P_2=\conv\{\eb_1,\ldots,\eb_{d-n}\}\subset \RR^{d-n}$ for some $1< n <d-1$.
    Therefore, we have $G=K_{n,d-n}$, and it is shown by \cite[Proposition 1.5]{hibi2021nearly} that for any $1< n <d-1$, $P_{K_{n,d-n}}$ is nearly Gorenstein if and only if $d-n\in \{n,n+1\}$.
    Since $P_{K_{n,n}}$ is Gorenstein, we get the desired result.
\end{proof}






\subsection{Nearly Gorenstein graphic matroid polytopes}

We start by giving one of several equivalent definitions of a matroid.
\begin{definition}
    Let $E$ be a finite set and let $\mathcal{B}$ be a subset of the power set of $E$ satisfying the following properties:
    \begin{enumerate}
        \item $\mathcal{B}\neq\varnothing$.
        \item If $A,B\in\mathcal{B}$ with $A\neq B$ and $a\in A\setminus B$, then there exists some $b\in B\setminus A$ such that $(A\setminus\set{a})\cup\set{b}\in\mathcal{B}$.
    \end{enumerate}
    Then the tuple $M=(E,\mathcal{B})$ is called a \emph{matroid} with \emph{ground set} $E$ and \emph{set of bases} $\mathcal{B}$.
\end{definition}
Let now $G=(V,E)$ be a multigraph.
The \emph{graphic matroid} associated to $G$ is the matroid $M_G$ whose ground set is the set of edges $E$ and whose bases are precisely the subsets of $E$ which induce a spanning tree of $G$.
Given two matroids $M_E=(E,\mathcal{B}_E)$ and $M_F=(F,\mathcal{B}_F)$, their \emph{direct sum} $M_E\oplus M_F$ is the matroid with ground set $E\sqcup F$ such that for each basis $B$ of $M_E\oplus M_F$, there exist bases $B_E\in\mathcal{B}_E$ and $B_F\in\mathcal{B}_F$ with $B=B_E\sqcup B_F$.
If such a decomposition is not possible for a matroid $M$, we call it \emph{irreducible}.

A graphic matroid with underlying multigraph $G$ is irreducible if and only if its underlying graph is $2$-connected.
If it is not irreducible, its irreducible components correspond precisely to the $2$-connected components of $G$.

For any matroid $M=(E,\mathcal{B})$, we can define its \emph{matroid base polytope} (or simply \emph{base polytope}) by
\[B_M=\conv\set{\sum_{b\in B} e_b\colon B\in\mathcal{B}} \subset \RR^{\abs{E}}\]
where $e_b$ is the incidence vector in $\RR^{\abs{E}}$ corresponding to the basis $b$.
If $B_M$ comes from a graphic matroid $M_G$, we will call it $B_G$.

An alternative definition of matroid base polytopes is as follows.
\begin{definition}[{\cite[Section~4]{gelfand1987combinatorial}}]
    A $(0,1)$-polytope $P\subset\RR^d$ is called \emph{(matroid) base polytope} if there is a positive integer $h$ such that every vertex $v=(v_1,\ldots,v_n)$ satisfies $\sum_{i=1}^d v_i = h$ and every edge (i.e. dimension $1$ face) of $P$ is a translation of a vector $e_i-e_j$ with $i\neq j$.
\end{definition}
It is shown in \cite[Theorem~4.1]{gelfand1987combinatorial} that this definition is indeed equivalent to that of a base polytope as given above and that the underlying matroid is uniquely determined.
This gives us the following two lemmas.
\begin{lemma}
    Let $G$ be a multigraph and let $G_1,\ldots G_n$ be its $2$-connected components.
    Then $B_G$ can be written as a direct product of the base polytopes $B_{G_1},\ldots,B_{G_n}$.
    Conversely, if $B_G$ can be written as a direct product of polytopes $P_1,\ldots,P_n$, where no $P_i$ is itself a direct product, then these polytopes correspond to the base polytopes of the $2$-connected components $G_1,\ldots,G_n$ of $G$.
\end{lemma}

\begin{proof}
    The first statement is trivially satisfied.

    The converse follows from two key insights.
    Firstly, the fact that if a base polytope $B_M$ associated to a (not necessarily graphic) matroid $M$ can be written as a direct product $P_1\times P_2$, then $P_1$ and $P_2$ are again base polytopes.
    Secondly, if a graphic matroid $M_G$ can be written as a direct sum $M_1\oplus M_2$, then $M_1$ and $M_2$ are again graphic matroids corresponding to subgraphs of $G$ which have at most one vertex in common.

    The first insight follows from the alternative definition of a base polytope:
    Every edge of $B_M$ is given by an edge in $P_1$ and a vertex of $P_2$, or vice versa.
    Hence, $P_1$ and $P_2$ must satisfy the definition as well, making them base polytopes with unique underlying matroids $M_1$ and $M_2$.
    The second insight is a classical result and can be found, among other places, in \cite[Lemma~8.2.2]{truemper1992matroid}.
\end{proof}

The following proposition is the polytopal version of a classical result due to White.

\begin{lemma}[{\cite[Theorem~1]{white1977basis}}]
    Matroid base polytopes are IDP.
\end{lemma}


We can now define Gorensteinness, nearly Gorensteinness, and levelness of a matroid by identifying it with its base polytope.
In \cite{hibi2021gorenstein} and \cite{kolbl2020gorenstein}, a constructive, graph-theoretic criterion of Gorensteinness for graphic matroids was found.
Since the direct product of two Gorenstein polytopes that have the same codegree is again Gorenstein, the characterisation is presented in terms of $2$-connected graphs.

\begin{proposition}[{\cite[Theorems~2.22 and 2.25]{kolbl2020gorenstein}}]\label{gormatroid}
    Let $G$ be a $2$-connected multigraph.
    Then the following are equivalent.
    \begin{enumerate}
        \item $B_G$ is Gorenstein with codegree $a=2$
        \item Either $G$ is the $2$-cycle or $G$ can be obtained from copies of the clique $K_4$ and Construction~2.15 in \cite{kolbl2020gorenstein}.
    \end{enumerate}
    The following are also equivalent.
    \begin{enumerate}
        \item $B_G$ is Gorenstein with codegree $a>2$
        \item $G$ can be obtained from copies of the cycle $C_{a}$ and Constructions~2.15,~2.17,~2.18 in \cite{kolbl2020gorenstein} with $\delta=a$.
    \end{enumerate}
\end{proposition}

The full characterisation of nearly Gorenstein graphic matroids is thus an immediate corollary of Theorem~\ref{01} and Proposition~\ref{gormatroid}.

\begin{corollary}\label{cor:graphic_matroids}
    Let $G$ be a multigraph with $2$-connected components $G_1,\ldots,G_n$, then the following are equivalent.
    \begin{enumerate}
        \item $B_G$ is nearly Gorenstein
        \item $B_{G_1},\ldots,B_{G_n}$ are Gorenstein with codegrees $a_1,\ldots,a_n$, where $\abs{ a_i - a_j} \leq 1$ for $1\leq i < j \leq s$.
    \end{enumerate}
\end{corollary}

\bibliographystyle{plain}
\bibliography{References}

\end{document}